\documentclass[11pt]{amsart}
\usepackage{amsmath,amsfonts,amssymb,graphics,color,enumerate}
\usepackage[applemac]{inputenc}
\usepackage{hyperref}
\usepackage{stmaryrd}
\usepackage{amsthm}
\usepackage{mathrsfs}
\usepackage{mathtools}
\usepackage{bbm}
\usepackage{graphicx}
\newtheorem{theorem}{Theorem}[section]
\newtheorem{lemma}[theorem]{Lemma}
\newtheorem{prop}[theorem]{Proposition}
\newtheorem{remark}[theorem]{Remark}

\def\div{\operatorname{div}}
\providecommand{\R}{\mathbb{R}}

\providecommand{\varepsilon}{}
\renewcommand{\leq}{\leqslant}
\renewcommand{\geq}{\geqslant}

\numberwithin{equation}{section}
\begin{document}
\date{\today}
\title[Remote trajectory tracking at low Reynolds number]{Remote trajectory tracking of a rigid body in an incompressible fluid at low Reynolds number}
%
% Authors
%
%
\author{J\'ozsef J. Kolumb\'an} 
\address{Institut f\"ur Mathematik, Universit\"at Leipzig, D-04109, Leipzig, Germany} 
%
%
% Abstract 
%
\begin{abstract}
In this paper we study  the motion of a rigid body driven by Newton's law immersed in a stationary incompressible Stokes flow occupying a bounded simply connected domain. The aim is that of trajectory tracking of the solid by the means of a control in the form of Dirichlet boundary data on the outside boundary of the fluid domain. We show that it is possible to exactly achieve any smooth trajectory for the solid that stays away from the external boundary, by the means of such a remote control. The proof relies on some density methods for the Stokes system, as well as a reformulation of the solid equations into an ODE.% (inspired by \cite{Flimitcurves}).
\end{abstract}
%
%                                   ====================================================
%
\maketitle
\tableofcontents
%
%
%
%
%
%
%%%%%%%%%%%%%%%%%%%%%%%%%%%%%%%%%%%%%%%%%%%%%%%%%%%%%%%%%%%%%%%%%%%%%%%%%%%%%%%%%%%%%%%%%%%%%%%%%%%%%%%%%%%%%%%%%%%%%%%%
%
%
%
%
%
\section{Introduction and statement of results}

We are interested in investigating the evolution of a rigid body immersed in an incompressible fluid at low Reynolds number, under the influence of an external boundary control. The model is given by the following coupled system consisting of the stationary incompressible Stokes equation for the fluid: 
\begin{align}\label{eq:solidstokes1}
\begin{split}
-\Delta u + \nabla p &=0\text{ in }\mathcal{F}(t),\\
\div u &=0\text{ in }\mathcal{F}(t),\\
u&=u_S\text{ on }\partial\mathcal{S}(t),\\
u&=g\text{ on }\partial\Omega,
\end{split}
\end{align}
and Newton's balance law for the solid:
\begin{align}\label{eq:solidstokes2}
\begin{split}
m h''(t)&=-\int_{\partial\mathcal{S}(t)}\Sigma(u,p)n\, d\sigma,\\
(\mathcal I(t) \theta'(t))'&=-\int_{\partial\mathcal{S}(t)}(x-h(t))\wedge\Sigma(u,p)n\, d\sigma,
\end{split}
\end{align}
for $t\in[0,T]$. Here we use the following notations: 
\begin{itemize}
\item $\Omega\subset \R^3$ is a bounded, open, simply connected domain, which at all times $t\in[0,T]$ is split into a 
non-empty, closed, regular, connected and simply connected set
$\mathcal{S}(t)$ occupied by the solid, and the remaining set $\mathcal{F}(t):=\Omega\setminus\mathcal{S}(t)$ occupied by the fluid;
\item $u$ denotes the fluid velocity, taking values in $\R^3$, while $p$ is the scalar-valued fluid pressure, the Cauchy stress tensor is the symmetric $3\times 3$ matrix given by
\begin{align}\label{eq:sigma}
\Sigma(u,p)=-p\text{Id}_3 +2D(u),\text{ where }D(u)=\frac{1}{2} \nabla u + (\nabla u) ^T\text{ is the symmetric gradient};
\end{align}
\item $n$ always denotes the unit normal vector on $\partial\mathcal{F}(t)$ pointing outside the fluid;
\item $g:[0,T]\times \partial\Omega\to\R^3$ plays the role of the control (as Dirichlet boundary data), and is supported on a fixed, non-empty open part $\Gamma$ of the outer boundary $\partial\Omega$, further satisfying the flux condition $\int_\Gamma g\cdot n\,d\sigma=0$;
\item the solid $\mathcal{S}(t)$ is completely characterized by its center of mass $h(t)\in\R^3$ and angle of rotation $\theta(t)\in\R^3$, its evolution is obtained through rigid movement via
\begin{align}
\mathcal{S}(t)=h(t)+R(\theta(t))(\mathcal S_0-h_0),
\end{align}
with
\begin{align}
R(\theta)=\begin{pmatrix}
\cos(\theta_1) & -\sin(\theta_1) & 0 \\
\sin(\theta_1)& \cos(\theta_1)& 0\\
0& 0& 1
\end{pmatrix}
\begin{pmatrix}
\cos(\theta_2)& 0 & -\sin(\theta_2) \\
0& 1& 0 \\
\sin(\theta_2)& 0& \cos(\theta_2)
\end{pmatrix}
\begin{pmatrix}
1&0& 0\\
0&\cos(\theta_3) & -\sin(\theta_3)  \\
0&\sin(\theta_3)& \cos(\theta_3)
\end{pmatrix}
\end{align}
being the standard three-dimensional rotation matrix;
\item one then has that
\begin{align}
\frac{d}{dt}R(\theta(t))=\theta'(t)\wedge R(\theta(t)),
\end{align}
where $\theta'(t)\wedge$ denotes the canonical skew-symmetric matrix associated with the vector $\theta'(t)$, 
and hence the solid velocity is given by
\begin{align}
u_S(t,x)=h'(t)+\theta'(t)\wedge(x-h(t))\text{ for }x\in\mathcal{S}(t);
\end{align}
\item $m>0$ denotes the mass of the solid, $\mathcal I(t)$ is the positive definite inertial matrix, evolving according to Sylvester's law:
\begin{align}
\mathcal I(t)=R(\theta(t))\mathcal I_0 R(\theta(t))^T.
\end{align}
\end{itemize}

The problem we are interested in is that of trajectory tracking, i.e. steering the solid along any given sufficiently regular trajectory by the means of the external boundary control $g$.

Let us denote $q:=(h,\theta),q':=(h',\theta')\in\R^6$. 
We observe that the dependence with respect to $t$ in $\mathcal F(t)$ comes only from the dependence with respect to $q$, hence from here on we will use the (slightly abusive) notations $\mathcal F(q(t))$,  $\mathcal S(q(t))$.
Furthermore, since we want to avoid the solid touching the outer boundary $\partial\Omega$, we will consider only positions $q$ from the set
$$\mathcal{Q}_\delta=\{q\in\R^6:\ \mathcal{S}(q)\subset\Omega,\ d(\mathcal{S}(q),\partial\Omega)\geq\delta\},$$
for $\delta>0$.

We introduce the space in which we look for the spacial part of the control $g(t,\cdot)$ as
$$ H^{1/2}_m(\Gamma):=\left\{\phi\in  H^{1/2}(\partial\mathcal F;\R^3):\ \text{supp }\phi\subset\Gamma,\ \int_{\Gamma} \phi \cdot n\, d\sigma =0\right\}.$$
Note that if $g\in C^0([0,T];H^{1/2}_m(\Gamma))$ and $q\in C^2([0,T];\mathcal Q_\delta)$ are given, system \eqref{eq:solidstokes1} has a unique solution $u\in C^0([0,T];H^1(\mathcal{F}(q(t)))) $, $p\in C^0([0,T];L^2(\mathcal{F}(q(t)))/\R) $, see for example \cite{tem}, where
$$L^2(A)/\R:=\left\{\phi\in L^2(A):\ \int_A \phi(x)\, dx=0\right\}\text{ for }A\subset\Omega.$$
We observe that there is a slight abuse of notation in writing $C^0([0,T];H^1(\mathcal{F}(q(t))))$ for a space of functions defined for each $t$ in the fluid domain $\mathcal{F}(q(t))$.
We will henceforth use the convention that for a functional space $X$ of functions depending on the variable $x$, the notation  $C^0([0,T]; X (\mathcal{F}(q(t))))$ signifies the space of functions defined for each $t$ in the fluid domain ${\mathcal F}(q(t))$, which can be extended to functions in $C^0([0,T];X(\R^{2}))$.

 Furthermore, if $\psi\in H^1(\mathcal{F}(q(t));\R^3)$ for some fixed $t\in[0,T]$ with $\div \psi=0$, one may test the Stokes equation in \eqref{eq:solidstokes1} with $\psi$ to obtain via integration by parts that
\begin{align}\label{eq:stibp}
0=\int_{\mathcal{F}(q(t))}(-\Delta u +\nabla p)\cdot \psi \, dx=2\int_{\mathcal{F}(q(t))}D(u):D(\psi) \, dx-\int_{\partial\mathcal{S}(q(t))}\Sigma(u,p)n\cdot \psi \, d\sigma.
\end{align}

\begin{figure}[h]
\includegraphics{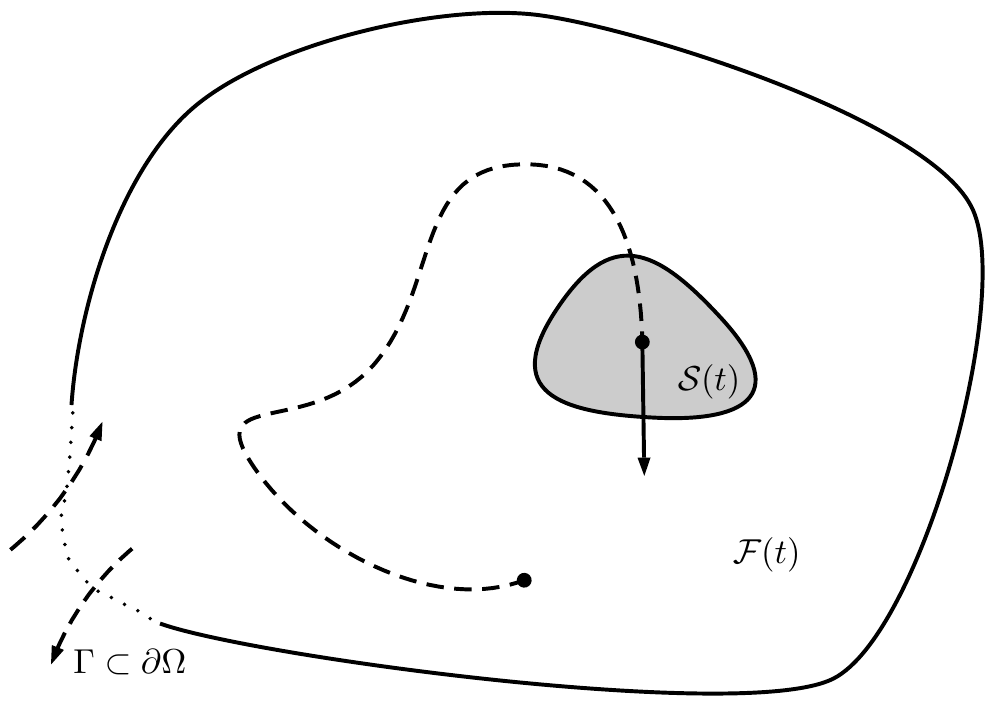}
\caption{The setting of the trajectory tracking problem.}
\end{figure}

The main result of the paper is the following trajectory tracking result for the solid movement.
\begin{theorem}\label{thm:main1}
For any $\delta >0$, there exists  a  finite dimensional subspace $\mathcal E$ of $H^{1/2}_m(\Gamma)$
such that
the following holds.
Let  $T>0$, $\mathcal K\subset \R^6\times\R^6$ compact, then there exists
 a control law
$\mathcal C \in \text{Lip}(\cup_{q\in\mathcal Q_\delta}\{q\}\times \mathcal K;\mathcal E),$
which satisfies that for  any given trajectory  $q$ in $C^{2} ([0,T] ; \mathcal Q_\delta)$ with $(q'(t),q''(t))\in\mathcal K$ for $t\in[0,T]$,
there exists a velocity field  $u\in C^0([0,T];H^1(\mathcal{F}(q(t)))) $  and pressure $p\in C^0([0,T];L^2(\mathcal{F}(q(t)))/\R) $
such that $(q,u,p)$ is the unique solution of
\eqref{eq:solidstokes1}-\eqref{eq:solidstokes2} with
$g(t,\cdot)=\mathcal C (q(t),q'(t),q''(t))(\cdot)$  for all $t$ in $[0,T]$. In addition, the cost of the control law $\mathcal C$ can be estimated by
\begin{align}\label{eq:cost}
\|\mathcal C (q,q',q'')(\cdot)\|_{H^{1/2}_m(\Gamma)}\leq C_\delta(|q''|+|q'|+|q'|^2),\ \forall (q,q',q'')\in\mathcal Q_\delta\times\mathcal K,
\end{align}
where $C_\delta>0$ only depends on $\mathcal{Q}_\delta$.
\end{theorem}

\textbf{Comparison to literature.}

Note that similarly to the trajectory tracking result from \cite{RTT} in the inviscid case, our control is of feedback form, depending on the solid position, velocity and acceleration. However, there is no circulation or vorticity involved since we consider the low Reynolds number case, hence the control only depends on the finite dimensional quantity $(q,q',q'')$. This can be an advantage in terms of possible practical implementation, compared to the discussion in the very last paragraph of Section 3 of \cite{RTT}, where the disadvantages of having a feedback law depending on the vorticity (as an infinite dimensional quantity) were presented.
Furthermore, the control takes values only in the finite-dimensional subspace $\mathcal E$ of $H^{1/2}_m(\Gamma)$, which can be a further advantage (in addition to the aforementioned feedback form) in such a direction of practical use. %of the control. 
However, one should note that in this paper we use completely different methods of proof than in \cite{RTT}, where the authors used a nonlinear
method to solve linear perturbations of nonlinear equations associated with a quadratic operator.
Our proof rather relies on reformulating the Newton equations into a quasi-linear ODE. We then prove exact trajectory tracking for this ODE by using a density argument for our Stokes system which can be seen as a generalization of the argument from Section 4 of \cite{GH}.
In addition, once more thinking in terms of practical implementation, one may note similarly as in \cite{GH}, that the density result Lemma \ref{lem:pou} from Section \ref{sec:ctrl} below can be advantageous in an approximation scheme. Finally, we note in Remark \ref{rmk:ctrl} at the end of the paper that in fact if the trajectory $q$ corresponds to the solution of the uncontrolled system on some subinterval of $[0,T]$, then the associated feedback control can be turned off on this subinterval. Once again in this context we may highlight the fact that although the fluid model is simpler than in \cite{RTT} (linear Stokes versus nonlinear Euler), we obtain more properties for our control law which can also be of interest from an engineering point of view.

Let us compare our result to the local null-controllability results obtained in the case of a rigid body evolving in a fluid modelled by the Navier-Stokes equations (with Dirichlet boundary conditions) from \cite{BG, BO, IT}. While our model is simpler due to the fact that we only consider the stationary Stokes equations (hence there is no time derivative or nonlinearity involved in the PDE part), if one thinks only on the level of the solid, our result of exact trajectory tracking is also stronger than merely local null-controllability. However, in the three papers mentioned above, the authors actually achieve  local null-controllability for the solid position  and the velocities of both the solid and the fluid, while in our current paper we only track the trajectory of the solid. In fact one should note that in our model there is always some fluid stuck to the solid boundary, so exactly controlling the fluid velocity in the whole domain at the same time as the prescribing the solid trajectory would run into difficulties. However, Lagrangian controllability of parts of the fluid away from the solid could be possible at the same time, using a strategy as in \cite{GH}.
On a different note, while the papers \cite{BG, BO, IT} rely on on Carleman estimates for the linearized equation, one can compare this to our use of a unique continuation argument for an appropriate Stokes system to prove the density result in Section \ref{sec:ctrl} below.

Let us also recall the difficult open problem of global exact controllability of the final position and velocity of a body evolving in a viscous fluid modelled by the Navier-Stokes equations with Dirichlet boundary conditions, due to the effect of boundary layers. However, if one considers the so-called Navier slip-with-friction boundary conditions instead, such a controllability has been achieved in \cite{K}, by generalizing certain methods introduced in \cite{CMS} in the case of a fluid alone.
However, as mentioned in \cite{K}, the methods used in that paper are not applicable in order to achieve trajectory tracking, since they rely on a time-rescale of the system.

In a different direction, one may also think of comparing our result to the controllability of swimmers (i.e. deformable bodies) moving in viscous fluids, where the control is no longer on the outside boundary of the fluid domain, but consists of the deformation of the solid itself. For the case where the fluid is modelled the Navier-Stokes equations (either stationary or non-stationary), see the papers  \cite{fish1, fish2, fish3, fish4}, while the case where the fluid is at low Reynolds number (hence modelled by the quasi-static Stokes equation) is handled in \cite{Al1,Al2,Loh,fish5}. In particular, some of the main steps of \cite{fish5} can be compared to those of our paper. For instance, the authors also use an ODE reformulation (which differs from ours due to the fact that their solid boundary deforms with the control, while ours does not), and their Lemma 4.1 can be paralleled to our Proposition \ref{prop:dens} and its use in Section \ref{sec:prmain1} to approximate certain elementary functions.

\textbf{Generalizations and open problems.}

A few remarks are in order regarding certain generalizations of Theorem \ref{thm:main1}. One may establish analogous results in the following cases: %when $\Omega$ is not simply connected, there are multiple solids evolving in the fluid instead of just one (see e.g. \cite{RTT}), there is a background flow present (see \cite{Flimitcurves}), or if one considers the problem in the two-dimensional setting. 
\begin{itemize}
\item if $\Omega$ is not simply connected;
\item if there are multiple solids evolving in the fluid instead of just one (see e.g. \cite{RTT});
\item if there is a background flow present (see \cite{Flimitcurves});
\item or if one considers the problem in the two-dimensional setting.
\end{itemize}
The respective assumptions were only made in order to simplify the presentation. Particularly in the view of tracking the trajectory of multiple solids,  in terms of possible practical applications, one can mention similar problems as those in \cite{RTT}, such as regrouping or dispersing solids, but in the case of low Reynolds number fluids. Such examples include cleaning up undesirable solids from the fluid or delivering some medication (solid particles) via an external fluid.

We further note the case of trajectory tracking for a solid in an unsteady Stokes flow as an interesting open problem. However, the extra $\partial_t u$ term seems incompatible with our current strategy, and for the moment we have not found any other approach to counteract this.

As mentioned above when comparing to the existing literature, trajectory tracking when the fluid is governed by the Navier-Stokes equations with Dirichlet boundary conditions is an even more challenging open problem.
 In fact, one should note that if one can prove exact trajectory tracking, then exact controllability of the position and velocity follows trivially, by picking an appropriate trajectory having the desired final position and velocity.

While we specifically introduced the set of admissible positions $\mathcal Q_\delta$ in order to avoid the solid touching the outer boundary $\partial\Omega$, one can also be interested in achieving controlled collisions in this sense. In such a direction one should mention the papers \cite{Hil1,Hil2} where the authors have shown that at least in the uncontrolled case, in the setting when the solid is a ball and the outer boundary is flat, there can be no collision (both in case of the Navier-Stokes and) in the case of the stationary Stokes system. However, a priori  this does not rule out the possibility that in some other controlled regime, one could make the solid touch the outer boundary.

\textbf{Structure of the paper.} In Section \ref{sec:reform} we reformulate the Newton equations into a quasi-linear ODE. In Section \ref{sec:ctrl} we show a density argument for controlled solutions of the Stokes system. Finally in Section \ref{sec:prmain1} we prove our main result by appropriately combining these two parts together.

\section{Reformulation of the solid equations}\label{sec:reform}

In this Section we will reformulate the system \eqref{eq:solidstokes1}-\eqref{eq:solidstokes2} by splitting the fluid velocity $u$ into two parts, one corresponding to the control and the other to the movement of the solid. For similar reformulation methods for rigid movement in Stokesian dynamics, see for instance the references \cite{Bossis} and \cite{brenner}.

More precisely, for $q\in\mathcal Q_\delta$, we introduce the so-called "elementary rigid velocities"
\begin{align}\label{eq:erv}
\phi_i(q,x)= \left\{
\begin{array}{ll}
      e_i\text{ for }i=1,2,3,\\
      e_{i-3}\wedge(x-h)\text{ for }i=4,5,6,
\end{array} 
\right. 
\end{align}
where $(e_1,e_2,e_3)$ is the canonical basis of $\R^3$. We may then associate the so-called "elementary Stokes solutions" as the unique smooth solutions $(V_i(q,\cdot),P_i(q,\cdot))$ of the problem
\begin{align}\label{eq:elemst}
\begin{split}
-\Delta_x V_i + \nabla_x P_i &=0\text{ in }\mathcal{F}(q),\\
\div_x V_i &=0\text{ in }\mathcal{F}(q),\\
V_i&=\phi_i\text{ on }\partial\mathcal{S}(q),\\
V_i&=0\text{ on }\partial\Omega,
\end{split}
\end{align}
for $i=1,\ldots,6$. We will also use the notation $V(q,\cdot):=(V_1,\ldots,V_6)(q,\cdot)$ and $P(q,\cdot):=(P_1,\ldots,P_6)(q,\cdot)$.

We further consider the contribution due to the control $g\in H^{1/2}_m(\Gamma)$, denoted by $(u^c[g],p^c[g])\in H^1(\mathcal{F}(q))\times( L^2(\mathcal{F}(q))/\R)$, defined as the solution of the problem
\begin{align}\label{eq:ctrlst}
\begin{split}
-\Delta u^c + \nabla p^c &=0\text{ in }\mathcal{F}(q),\\
\div u^c &=0\text{ in }\mathcal{F}(q),\\
u^c&=0\text{ on }\partial\mathcal{S}(q),\\
u^c&=g\text{ on }\partial\Omega.
\end{split}
\end{align}
Clearly $g\mapsto(u^c[g],p^c[g])$ is linear.

It is then easy to see that the solution $(u,p)$ of \eqref{eq:solidstokes1} can be written as
\begin{align}\label{eq:split}
u(t,x)=q'(t)\cdot V(q(t),x) + u^c(t,x),\quad p(t,x)=q'(t)\cdot P(q(t),x) + p^c(t,x).
\end{align} 
We may further define the mass matrix
\begin{align}\label{eq:mass}
\mathcal M(q):=\begin{pmatrix}
m\text{Id}_3 & 0_3 \\
0_3 & \mathcal I(q),
\end{pmatrix}
\end{align}
as well as the so-called "Stokes resistance matrix"
\begin{align}\label{eq:stresm}
\mathcal K(q):=\left(\int_{\partial\mathcal S(q)}\Sigma(V_i(q,\cdot),P_i(q,\cdot))n \cdot \phi_j(q,\cdot)\,d\sigma\right)_{i,j=1,\ldots,6}=2\left(\int_{\mathcal F(q)}D(V_i(d,\cdot)):D(V_j(d,\cdot))\,dx\right)_{i,j=1,\ldots,6},
\end{align}
where we have used an integration by parts similar to \eqref{eq:stibp} to obtain the second equality.
It can be checked that
$\mathcal K(q)$ is symmetric, positive definite and invertible, and the map $q\in\mathcal Q_\delta\mapsto\mathcal K(q)$ is Lipschitz (see Section 5 in \cite{Flimitcurves}).

Then, using \eqref{eq:sigma} and \eqref{eq:split}, the Newton equations \eqref{eq:solidstokes2} become 
\begin{align}\label{eq:ode}
\frac{d}{dt}(\mathcal M(q) q')=- \mathcal K(q) q' -\left(\int_{\partial\mathcal S(q)}\Sigma(u^c,p^c)n \cdot \phi_i(q,\cdot)\,d\sigma\right)_{i=1,\ldots,6}
\end{align}
Note that if $g=0$, this reduces to a quasi-linear ordinary differential equation.

\section{A density result for controlled Stokes systems}\label{sec:ctrl}

In this Section we generalize the density result Theorem 4.1 from \cite{GH} to our setting.

Let us consider the following simplified model where we neglect the solid displacement. We fix some $q\in\mathcal{Q}_\delta$ and
introduce the space
\begin{align}
H^{1/2}_m(\partial\mathcal{S}(q)):&=\left\{\phi\in  H^{1/2}(\partial\mathcal{S}(q);\R^3):\  \int_{\partial\mathcal{S}(q)} \phi \cdot n\, d\sigma =0\right\},
\end{align} 
and look at the problem
\begin{align}\label{eq:modst}
\begin{split}
-\Delta u +\nabla p&=0\text{ in }\mathcal F(q),\\
\div u &= 0\text{ in }\mathcal F(q),\\
u&=0\text{ on }\partial\mathcal S(q),\\
u&=\alpha\text{ on }\partial\Omega,
\end{split}
\end{align}
for some given $\alpha\in H^{1/2}_m(\Gamma)$.

The main result of this section is the following proposition.

\begin{prop}\label{prop:dens}
For any $q\in\mathcal Q_\delta$, the set
$$\{\Sigma(u,p)n|_{\partial\mathcal S(q)}:\ (u,p)\text{ solves }\eqref{eq:modst} \text{ for some }\alpha\in H^{1/2}_m(\Gamma)\}$$
is dense in $ (H^{1/2}_m(\partial\mathcal S(q)))'$.
\end{prop}
\begin{proof}
We shall prove that
\begin{align}\label{eq:zset}
\left\{w\in H^{1/2}_m(\partial\mathcal S(q)):\text{ for any } \alpha\in H^{1/2}_m(\Gamma)\text{ the solution }(u,p)\text{ of }\eqref{eq:modst}\text{ satisfies }\int_{\partial\mathcal S(q)}\Sigma(u,p)n\cdot w\, d\sigma=0\right\}\end{align}
contains only $0$,
from which the claim follows by duality.

With any $w$ in the above set, we may associate the solution of the following system:
\begin{align}\label{eq:vst}
\begin{split}
-\Delta v +\nabla q&=0\text{ in }\mathcal F(q),\\
\div v &= 0\text{ in }\mathcal F(q),\\
v&=w\text{ on }\partial\mathcal S(q),\\
v&=0\text{ on }\partial\Omega,
\end{split}
\end{align}
which satisfies $v\in H^1(\mathcal F(q))$.

Therefore, for given $\alpha\in H^{1/2}_m(\Gamma)$, we may test \eqref{eq:vst} with the solution $u$ of \eqref{eq:modst} to obtain using integration by parts that
\begin{align*}
2\int_{\mathcal F(q)} D(u):D(v)\, dx=\int_{\Gamma}\Sigma(v,q)n \cdot \alpha \, d\sigma.
\end{align*}
On the other hand, testing \eqref{eq:modst} with $v$ similarly gives
\begin{align*}
2\int_{\mathcal F(q)} D(u):D(v)\, dx=\int_{\partial\mathcal S(q)}\Sigma(u,p)n\cdot w \, d\sigma=0,
\end{align*}
due to $w$ belonging to \eqref{eq:zset}.

Hence we deduce that
\begin{align*}
\int_{\Gamma}\Sigma(v,q)n \cdot \alpha \, d\sigma=0,\ \forall 
\alpha\in H^{1/2}_m(\Gamma),
\end{align*}
from where it follows that $\Sigma(v,q)n=0$ in $ (H^{1/2}_m(\Gamma))'$.

By unique continuation for the Stokes-system (see for instance Corollary 1.1 in \cite{Fabre} or Corollary 1.2 in \cite{BEG}), since $v=0$ on $\Gamma$, it follows that $v=0$ in $\mathcal F(q)$, and hence $w=0$. 

\end{proof}

\section{Proof of Theorem \ref{thm:main1}}\label{sec:prmain1}

We are now in position to prove our main result regarding the trajectory tracking of a solid immersed in a stationary Stokes flow.

In order to obtain a finite family of controls to span the subspace $\mathcal E$ we are looking for,
we will make use of the following lemma.

\begin{lemma}\label{lem:pou}
Let $\delta>0$, $\varepsilon>0$, there exists a finite dimensional subspace $\mathcal E_\varepsilon$ of $H^{1/2}_m(\Gamma)$ and Lipschitz mappings 
$$
q\in\mathcal Q_\delta\mapsto \bar  g^\varepsilon_i(q,\cdot)\in\mathcal E_\varepsilon\text{ for }i=1,\ldots,6,$$ such that for any $q\in\mathcal Q_\delta$ there holds
\begin{align}\label{eq:close}
\left|\left(\int_{\partial\mathcal S(q)}\Sigma(u^c[\bar g^\varepsilon_i(q,\cdot)],p^c[\bar g^\varepsilon_i(q,\cdot)])n \cdot \phi_j(q,\cdot)\,d\sigma\right)_{i,j=1,\ldots,6}-\mathcal K(q)\right| \leq\varepsilon.
\end{align}
\end{lemma}
\begin{proof}
For any $q\in\mathcal Q_\delta$, $\varepsilon>0$, $i\in\{1,\ldots,6\}$ one may apply Proposition \ref{prop:dens} 
%with $\mathcal F=\mathcal F(q)$, $\mathcal F=\mathcal S(q)$
 to deduce that there exist $g^\varepsilon_i(q,\cdot)\in H^{1/2}_m(\Gamma)$ such that
$$\|\Sigma(V_i(q,\cdot),P_i(q,\cdot))n-\Sigma(u^c[g^\varepsilon_i(q,\cdot)],p^c[g^\varepsilon_i(q,\cdot)])n\|_{ (H^{1/2}_m(\partial\mathcal S(q)))'}\leq \varepsilon,$$
where we recall that $V_i$ was defined in \eqref{eq:elemst}, respectively $u^c[\cdot]$ in \eqref{eq:ctrlst}. Using integration by parts we may see that the functions $\phi_i$ defined in \eqref{eq:erv} are actually in $H^{1/2}_m(\partial\mathcal{S}(q))$, so we may then conclude that
\begin{align}\label{eq:intclose}
\left|\left(\int_{\partial\mathcal S(q)}\Sigma(u^c[g^\varepsilon_i(q,\cdot)],p^c[g^\varepsilon_i(q,\cdot)])n \cdot \phi_j(q,\cdot)\,d\sigma\right)_{i,j=1,\ldots,6}-\mathcal K(q)\right| \leq C \varepsilon,
\end{align}
for some $C>0$ independent of $q\in\mathcal Q_\delta$.

We may associate for $q,\tilde q\in\mathcal Q_\delta$, the unique solution $(\tilde u,\tilde p)\in H^1\times L^2/\R$ of the Stokes problem
\begin{align*}
-\Delta \tilde u(q,\tilde q,\cdot) + \nabla \tilde p(q,\tilde q,\cdot) &=0\text{ in }\mathcal{F}(\tilde q),\\
\div \tilde u(q,\tilde q,\cdot) &=0\text{ in }\mathcal{F}(\tilde q),\\
\tilde u(q,\tilde q,\cdot)&=0\text{ on }\partial\mathcal{S}(\tilde q),\\
\tilde u(q,\tilde q,\cdot)&=g^\varepsilon_i(q,\cdot)\text{ on }\partial\Omega.
\end{align*}
Arguing once more as in Section 5 of \cite{Flimitcurves}, one can obtain that the map
$$\tilde q\in\mathcal Q_\delta\mapsto \int_{\mathcal F(\tilde q)} D(\tilde u(q,\tilde q,\cdot)):D(V_i(\tilde q,\cdot)) \, dx $$
is Lipschitz. Using integration by parts similarly to \eqref{eq:stibp}, the fact that $q\in\mathcal Q_\delta\mapsto \mathcal K(q)$ is Lipschitz, and the estimate \eqref{eq:intclose}, we may deduce that for any $q\in\mathcal Q_\delta$, there exists $r_q=r_q(\varepsilon)>0$ such that

\begin{align}\label{eq:intclose2}
\left|\left(\int_{\partial\mathcal S(q)}\Sigma(\tilde u(\hat q,q,\cdot),\tilde p(\hat q,q,\cdot))n \cdot \phi_j(q,\cdot)\,d\sigma\right)_{i,j=1,\ldots,6}-\mathcal K(q)\right| \leq \tilde C \varepsilon,\text{ for }|q-\hat q|<r_q,
\end{align}
where once more $\tilde C>0$ is independent of $q\in\mathcal Q_\delta$.

Using the compactness of $\mathcal Q_\delta$, one may extract a finite subcover of balls $\{B(q_l,r_l)\}_{l=1,\ldots,N_\delta}$, and use a partition of unity $\{\rho_l\}_{l=1,\ldots,N_\delta}$ adapted to this subcover to obtain that
$$\bar g_i^\varepsilon(q,\cdot):=\sum_{l=1}^{N_\delta}\rho_l(q)g^\varepsilon_i(q_l,\cdot)$$
satisfies \eqref{eq:close} with $\hat C\varepsilon$ instead of $\varepsilon$ on the right-hand side. Reparametrizing with respect to $\varepsilon$ and setting $\mathcal E_\varepsilon:=\text{span}\{g^\varepsilon_i(q_l,\cdot):\ i=1,\ldots,6,\ l=1,\ldots,N_\delta\}$ allows us to finish the proof of the lemma.

\end{proof}

We then continue working with the functions $\bar g_i^\varepsilon$ given by Lemma \ref{lem:pou}.
Since the Stokes resistance matrix defined in \eqref{eq:stresm} is invertible, there exists $\bar\varepsilon>0$ small enough such that the matrix
\begin{align}\label{eq:ctrlmat}
\left(\int_{\partial\mathcal S(q)}\Sigma(u^c[\bar g^{\bar\varepsilon}_i],p^c[\bar g^{\bar\varepsilon}_i])n \cdot \phi_j(q,\cdot)\,d\sigma \right)_{i,j=1,\ldots,6}
\end{align}
is also invertible. 
Note that a priori  $\bar\varepsilon$ would depend on $q$, but since $q\in\mathcal Q_\delta\mapsto\mathcal K (q)$ is continuous and $Q_\delta$ is compact, $\bar\varepsilon>0$ can be chosen to be uniform for $q\in\mathcal Q_\delta$.

We will then take the finite dimensional subspace $\mathcal E$ we are looking for to be $\mathcal E_{\bar\varepsilon}$.
We may look for a control of the form
$$g(q,\cdot)=\sum_{i=1}^6\mu_i \bar g_i^{\bar\varepsilon}(q,\cdot),$$
where $\mu\in\R^6$. Plugging this into \eqref{eq:ode} and using the linearity of $\Sigma(u^c[\cdot],p^c[\cdot])$ with respect to $g$, the trajectory tracking problem reduces to finding a vector $\mu$ such that
\begin{align*}
\frac{d}{dt}(\mathcal M(q) q')+ \mathcal K(q) q'=-\left(\int_{\partial\mathcal S(q)}\Sigma(u^c[\bar g^{\bar\varepsilon}_i],p^c[\bar g^{\bar\varepsilon}_i])n \cdot \phi_j(q,\cdot)\,d\sigma \right)_{i,j=1,\ldots,6} \mu.
\end{align*}
However, since we have shown above that the matrix in \eqref{eq:ctrlmat} is invertible, one deduces the existence of such $\mu=\mu(q,q',q'')$ as given by
\begin{align}\label{eq:mudef}
\mu(q,q',q'')=-\left(\int_{\partial\mathcal S(q)}\Sigma(u^c[\bar g^{\bar\varepsilon}_i],p^c[\bar g^{\bar\varepsilon}_i])n \cdot \phi_j(q,\cdot)\,d\sigma \right)_{i,j=1,\ldots,6}^{-1}\left( \frac{d}{dt}(\mathcal M(q) q')+ \mathcal K(q) q'\right).\end{align}
It is further easy to check that $\mu$ is Lipschitz with respect to $(q,q',q'')\in\mathcal Q_\delta\times\mathcal K$. This concludes the proof of the existence part of Theorem \ref{thm:main1} by setting $$\mathcal C(q,q'q'')(\cdot):=\sum_{i=1}^6\mu(q,q',q'') \bar g_i^{\bar\varepsilon}(q,\cdot)).$$
Note that for any given control $g$, the uniqueness of system \eqref{eq:solidstokes1}-\eqref{eq:solidstokes2} follows from the reformulation given in \eqref{eq:split} and \eqref{eq:ode}.

Finally, the estimate \eqref{eq:cost} on the cost of the control follows from \eqref{eq:mudef}, the form of $\mathcal M$ in \eqref{eq:mass}, the boundedness of $\mathcal K$, and the boundedness of $\bar g_i^\varepsilon$ from Lemma \ref{lem:pou}, by writing
\begin{align*}
\|\mathcal C (q,q',q'')(\cdot)\|_{H^{1/2}_m(\Gamma)}\leq C_\delta\left| \frac{d}{dt}(\mathcal M(q) q')+ \mathcal K(q) q'\right |,
\end{align*}
which concludes the proof of Theorem \ref{thm:main1}.

\begin{remark}[Switching off the control]\label{rmk:ctrl}
Note that if the target trajectory $q$ satisfies the uncontrolled ODE $$ \frac{d}{dt}(\mathcal M(q) q')+ \mathcal K(q) q'=0$$ on some subinterval of $[0,T]$, then by construction the corresponding control $g$ given by $\mathcal C(q,q',q'')$ in fact vanishes for all times in this subinterval.
\end{remark}

%
%
%
%########################################################
%

\def\cprime{$'$}

\end{document}